\documentclass[reqno, 12pt]{amsart}
\pdfoutput=1
\makeatletter
\let\origsection=\section \def\section{\@ifstar{\origsection*}{\mysection}} 
\def\mysection{\@startsection{section}{1}\z@{.7\linespacing\@plus\linespacing}{.5\linespacing}{\normalfont\scshape\centering\S}}
\makeatother        

\usepackage{amsmath,amssymb,amsthm}
\usepackage{mathrsfs}
\usepackage{mathabx}\changenotsign
\usepackage{dsfont}
 
\usepackage{xcolor}
\usepackage[backref]{hyperref}
\hypersetup{
    colorlinks,
    linkcolor={red!60!black},
    citecolor={green!60!black},
    urlcolor={blue!60!black}
}

\usepackage[open,openlevel=2,atend]{bookmark}

\usepackage[abbrev,msc-links,backrefs]{amsrefs} 
\usepackage{doi}

\renewcommand{\PrintDOI}[1]{\doi{#1}}

\usepackage[T1]{fontenc}
\usepackage{lmodern}
\usepackage[babel]{microtype}
\usepackage[english]{babel}

\usepackage{geometry}
\numberwithin{equation}{section}
\numberwithin{figure}{section}

\usepackage{enumitem}
\def\rmlabel{\upshape({\itshape \roman*\,})}

\def\alabel{\upshape({\itshape \alph*\,})}

\let\polishlcross=\l
\def\l{\ifmmode\ell\else\polishlcross\fi}

\def\qand{\quad\text{and}\quad}

\let\setminus=\smallsetminus

\let\sm=\setminus

\makeatletter
\def\moverlay{\mathpalette\mov@rlay}
\def\mov@rlay#1#2{\leavevmode\vtop{   \baselineskip\z@skip \lineskiplimit-\maxdimen
   \ialign{\hfil$\m@th#1##$\hfil\cr#2\crcr}}}
\newcommand{\charfusion}[3][\mathord]{
    #1{\ifx#1\mathop\vphantom{#2}\fi
        \mathpalette\mov@rlay{#2\cr#3}
      }
    \ifx#1\mathop\expandafter\displaylimits\fi}
\makeatother

\newcommand{\dcup}{\charfusion[\mathbin]{\cup}{\cdot}}

\DeclareFontFamily{U}  {MnSymbolC}{}
\DeclareSymbolFont{MnSyC}         {U}  {MnSymbolC}{m}{n}
\DeclareFontShape{U}{MnSymbolC}{m}{n}{
    <-6>  MnSymbolC5
   <6-7>  MnSymbolC6
   <7-8>  MnSymbolC7
   <8-9>  MnSymbolC8
   <9-10> MnSymbolC9
  <10-12> MnSymbolC10
  <12->   MnSymbolC12}{}
\DeclareMathSymbol{\powerset}{\mathord}{MnSyC}{180}

\usepackage{tikz}
\usetikzlibrary{shapes,snakes}
\usetikzlibrary{graphs,graphs.standard}
\usetikzlibrary{decorations.markings}
\usetikzlibrary{positioning,calc}

\usepackage{graphicx}
\usepackage{wrapfig}

\usepackage{subcaption}
\let\epsilon=\varepsilon
\let\eps=\epsilon
\let\rho=\varrho
\let\theta=\vartheta
\let\kappa=\varkappa

\def\RR{{\mathds R}}

\def\RT{\mathrm{RT}}
\def\ex{\mathrm{ex}}

\newcommand{\ccF}{\mathscr{F}}

\theoremstyle{plain}
\newtheorem{thm}{Theorem}[section]
\newtheorem{lemma}[thm]{Lemma}

\theoremstyle{definition}
\newtheorem{dfn}[thm]{Definition}

\let\lra=\longrightarrow
\def\wq{\widetilde{w}}
\let\phi=\varphi

\begin{document}

\title[Weighted variants of the Andr\'asfai-Erd\H{o}s-S\'os Theorem]
{Weighted variants of the Andr\'asfai-Erd\H{o}s-S\'os Theorem}

\author[Clara~M.~L\"uders]{Clara Marie L\"uders} 
\author[Christian Reiher]{Christian Reiher}
\thanks{The second author was supported by the European Research Council 
(ERC grant PEPCo 724903).}
\address{Fachbereich Mathematik, Universit\"at Hamburg, Hamburg, Germany}
\email{Christian.Reiher@uni-hamburg.de}
\email{Clara.Marie.Lueders@gmail.com}

\subjclass[2010]{05C35, 05C22}
\keywords{Tur\'an problems, weighted graphs, Andr\'asfai-Erd\H{o}s-S\'os theorem}

\begin{abstract}
	A well known result due to Andr\'asfai, Erd\H{o}s, and S\'os asserts that for $r\ge 2$
	every $K_{r+1}$-free graph $G$ on $n$ vertices with $\delta(G)>\frac{3r-4}{3r-1}n$ is
	$r$-partite. We study related questions in the context of weighted graphs, which are 
	motivated by recent work on the Ramsey-Tur\'an problem for cliques.
\end{abstract}

\maketitle

\section{Introduction}

\subsection{Simple graphs} \label{subsec:simple}

Extremal graph theory began with Tur\'an's discovery~\cite{Turan} that for $r\ge 2$ 
every $n$-vertex graph $G$ with more than $\frac{r-1}{r}\cdot \frac{n^2}{2}$ edges 
contains a $K_{r+1}$, i.e., a clique on~$r+1$ vertices. The constant $\frac{r-1}{r}$ appearing 
here is optimal, as can be seen by looking at balanced complete $r$-partite graphs. 
Simonovits~\cite{Si68} proved that this extremal configuration 
is subject to a stability phenomenon roughly saying that a $K_{r+1}$-free graph with 
almost the maximum number of edges is ``almost'' $r$-partite.  

\begin{thm}[Simonovits] \label{thm:stab}
	For every $r\ge 2$ and $\eps>0$ there exists some $\delta>0$ such that 
	every $K_{r+1}$-free graph $G$ on $n$ vertices with at least 
	$\bigl(\tfrac{r-1}{r}-\delta\bigr)\frac{n^2}{2}$ edges admits a partition 
	$V(G)=W_1\dcup \ldots \dcup W_r$ satisfying $\sum_{i=1}^r e(W_i)<\eps n^2$. 
\end{thm} 
      
A standard proof of Theorem~\ref{thm:stab} starts with the observation that an iterative
deletion of vertices with small degree allows us to reduce to the case that 
$\delta(G)>\bigl(\tfrac{r-1}{r}-\eta\bigr)n$ holds for an arbitrary constant $\eta\ll \eps$
chosen in advance.
One then takes a clique of order $r$ in $G$, whose existence is guaranteed by Tur\'an's
theorem, and observes that the joint neighbourhoods $\widetilde{W}_1, \ldots, \widetilde{W}_r$
of the $(r-1)$-subsets of this clique are mutually disjoint independent sets, since otherwise
$G$ would contain a $K_{r+1}$. Finally, the minimum degree condition ensures that these 
sets cover all but at most $r\eta n$ vertices of $G$, for which reason any partition 
$V(G)=W_1\dcup \ldots \dcup W_r$ with $W_i\supseteq \widetilde{W}_i$ for all $i\in [r]$
has the desired property. 

As a matter of fact, however, the second part of the argument can be omitted by appealing 
instead to a result of Andr\'asfai, Erd\H{o}s, and S\'os~\cite{AES74} telling us that an 
appropriate minimum degree condition of the form $\delta(G)>\bigl(\tfrac{r-1}{r}-\eta\bigr)n$
implies that $K_{r+1}$-free graphs are $r$-partite. For an elegant alternative proof of this 
fact we refer to Brandt~\cite{Brandt}.

\begin{thm}[Andr\'asfai, Erd\H{o}s, and S\'os]\label{thm:AES}
	Let $G$ be for some $r\ge 2$ a $K_{r+1}$-free graph on~$n$ vertices 
	satisfying $\delta(G)>\frac{3r-4}{3r-1}n$. Then there is a homomorphism
	from $G$ to $K_r$, i.e.,~$G$ is $r$-colourable.
\end{thm}

The constant $\frac{3r-4}{3r-1}$ appearing here is optimal.
This can be seen by constructing for some $n$ divisible by $(3r-1)$ a graph $G$
on a set $V$ of $n$ vertices having a partition 
\[
	V=A_1\dcup\ldots\dcup A_5\dcup B_1\dcup\ldots\dcup B_{r-2}
\]
such that 
\[
	|A_1|=\ldots=|A_5|=\tfrac{n}{3r-1}
	\quad \text{ and } \quad
	|B_1|=\ldots=|B_{r-2}|=\tfrac{3n}{3r-1}\,,
\]
and whose set of edges is as follows:
\begin{enumerate}
	\item[$\bullet$] there are all edges from a vertex in $A_i$ to a vertex in $A_{i+1}$, 
		where the indices are taken modulo~$5$;
	\item[$\bullet$] all edges from $A_i$ to $B_j$, where $i\in [5]$, $j\in [r-2]$;
	\item[$\bullet$] and all edges from $B_j$ to $B_{j'}$, where $j, j'\in [r-2]$ are 
		distinct.
\end{enumerate}
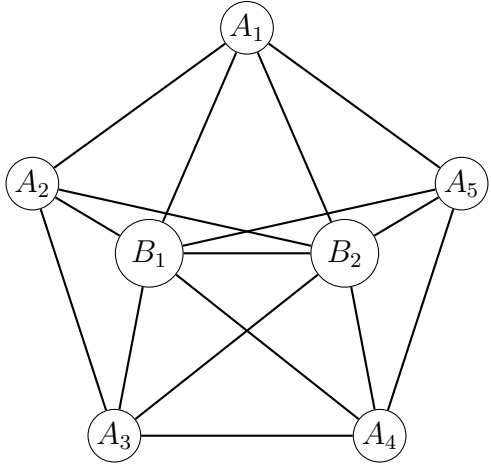
\begin{figure}[ht]
    \centering      
\begin{tikzpicture}[scale = 1]
	\node[circle,fill=white,draw,minimum size=0.7cm, inner sep=0pt] (x1) at (90:3) {$A_1$};
	\node[circle,fill=white,draw,minimum size=0.7cm, inner sep=0pt] (x2) at (162:3) {$A_2$};
	\node[circle,fill=white,draw,minimum size=0.7cm, inner sep=0pt] (x3) at (234:3) {$A_3$};
	\node[circle,fill=white,draw,minimum size=0.7cm, inner sep=0pt] (x4) at (306:3) {$A_4$};
	\node[circle,fill=white,draw,minimum size=0.7cm, inner sep=0pt] (x5) at (18:3) {$A_5$};

	\node[circle,fill=white,draw,minimum size=0.9cm, inner sep=0pt] (y1) at (180:1.3) {$B_1$};
	\node[circle,fill=white,draw,minimum size=0.9cm, inner sep=0pt] (y2) at (0:1.3) {$B_2$};
	
	\path[black,thick]
		(x1) edge (x2)
		(x2) edge (x3)
		(x3) edge (x4)
		(x4) edge (x5)
		(x5) edge (x1)
		(y1) edge (x1)
		(y1) edge (x2)
		(y1) edge (x3)
		(y1) edge (x4)
		(y1) edge (x5)
		(y2) edge (x1)
		(y2) edge (x2)
		(y2) edge (x3)
		(y2) edge (x4)
		(y2) edge (x5)
		(y1) edge (y2);	
\end{tikzpicture}
        \caption{The extremal graph for the case $r=4$ of Theorem~\ref{thm:AES}.}
\end{figure}

\subsection{Weighted graphs}
The concepts and problems discussed in the previous subsection make sense 
in the broader context of weighted graphs as well. For the purposes of this article, these
are defined as follows.

\begin{dfn}
	A {\it weighted graph} is a pair $G=(V, w)$ consisting of a finite {\it vertex set} $V$
	and a symmetric function $w\colon V^2\lra \RR_{\ge 0}$ such that $w(x, x)=0$ holds for
	all $x\in V$. 
\end{dfn}

Here the word ``symmetric'' means that we require $w(x, y)=w(y, x)$ for 
all $x, y\in V$.
The notions of subgraphs, induced subgraphs, and isomorphisms extend in the following way
from ordinary graphs to weighted graphs.

\begin{dfn}
	Let $G=(V, w)$ and $G'=(V', w')$ be two weighted graphs. We say that~$G'$ is a 
	{\it subgraph} of~$G$ if $V'\subseteq V$ and, additionally, $w'(x, y)\le w(x, y)$ 
	holds for all $x, y\in V'$. If this conditions holds with equality throughout 
	we call $G'$ an {\it induced subgraph} of $G$. Finally,~$G$ and~$G'$ are said to be 
	{\it isomorphic} if there is a bijection $\phi\colon V\lra V'$ satisfying 
	$w(x, y)=w'\bigl(\phi(x), \phi(y)\bigr)$ for all $x, y\in V$.
\end{dfn}

For two weighted graphs $F$ and $G$ we say that $G$ is {\it $F$-free} if $G$ does not possess
any subgraph isomorphic to $F$. More generally, if $\ccF$ is a set of weighted graphs
such that $G$ is $F$-free for every $F\in \ccF$, then $G$ is said to be {\it $\ccF$-free}.
The natural analogue of the ``number of edges'' of a weighted graph $G=(V, w)$ is, of course, 
the quantity 
\[
	e(G)=\tfrac 12 \sum_{(x, y)\in V^2}w(x, y)\,.
\]

Now for every such set $\ccF$ of weigthed graphs and every finite set $D\subseteq \RR_{\ge 0}$
one may look at the {\it extremal function} $n\longmapsto \ex_D(n, \ccF)$ sending every 
positive integer $n$ to the maximum of $e(G)$ as $G$ varies over $\ccF$-free weighted 
graphs of order $n$ whose weight function only attains values in $D$. 
The natural generalisation of Tur\'an's problem to this context asks to determine 
these functions for all choices of $\ccF$ and $D$, the classical 
case being $D=\{0, 1\}$. 

Similar as in this case, the {\it generalised Tur\'an densities} 
\begin{equation}\label{eq:wT}
	\pi_D(\ccF)=\lim_{n\to\infty}\frac{\ex_D(n, \ccF)}{n^2/2}
\end{equation}
are easily shown to exist. Questions concerning $\ex_D(n, \ccF)$ and $\pi_D(\ccF)$ are often 
studied in the literature, both for their own sake (see e.g.~\cites{FK02, RoSi95}) and 
due to their connection with other parts of extremal combinatorics. 

For instance, De Caen and F\"{u}redi~\cite{DeFu00} realised that such results 
can be applied to determine the Tur\'an density of the Fano plane. To this end, they 
needed to know the value of $\pi_D(\ccF)$, where $D=\{0, 1, 2, 3, 4\}$ and $\ccF$ denotes 
the set of all corresponding weighted graphs~$F$ on four vertices with $e(F)\ge 21$. 
Their approach is occasionally called the {\it link multigraph method}
and led to many further results on Tur\'an's hypergraph problem 
(see also~\cites{BR, FuSi05, KeMu12, KeSu05, MuRo02}).   

An earlier occurrence of a Tur\'an problem for weighted graphs appeared in the 
determination of the so-called {\it Ramsey-Tur\'an density} of even cliques 
due to Erd\H{o}s, Hajnal, Szemer\'edi, and S\'os~\cite{EHSS}. This result belongs to 
an area initiated by Vera T. S\'os, which is called {\it Ramsey-Tur\'an theory}. Given
a graph $F$, a number $n$ of vertices, and a real number~$m>0$ she defined the 
{\it Ramsey-Tur\'an number} $\RT(n, m, F)$ to be the maximum number of edges that 
an $F$-free graph $G$ on $n$ vertices with $\alpha(G)<m$ can have. The problem 
is especially interesting if $F$ is a clique and it is customary in this setting to 
pass to the {\it Ramsey-Tur\'an density function} $f_t\colon (0, 1)\to \RR$ defined by 
\[
	f_t(\delta)=\lim_{n\to\infty} \frac{\RT(n, \delta n, K_t)}{n^2/2}\,.
\]
A further simplification can be obtained by restricting the attention to the 
{\it Ramsey-Tur\'an densities}
\[
	\rho(K_t)=\lim_{\delta\to 0} f_t(\delta)\,.
\]

Such quantities have been intensively studied in the literature (see 
e.g. \cites{BE76, EHSS, ES69, Sz72} for important milestones and~\cite{SS01} 
for a beautiful survey). Owing to all these efforts it is known that
\begin{equation} \label{eq:rhot}
	\rho(K_t)=\begin{cases}
		\frac{t-3}{t-1} & \text{ if $t\ge 3$ is odd,} \cr
		\frac{3t-10}{3t-4} & \text{ if $t\ge 4$ is even.}
		\end{cases}
\end{equation}
The even case is much harder and in their solution Erd\H{o}s, Hajnal, Szemer\'edi, and S\'os
applied a result on the Tur\'an density of a certain $\{0, 1, 2\}$-valued collection 
$\ccF_t$ of weighted graphs to a reduced graph obtained by means of Szemer\'edi's 
regularity lemma~\cite{Sz78}. These specific families $\ccF_t$ of weighted graphs are 
introduced in Definiton~\ref{dfn:Gti} below.

A few years ago, Fox, Loh, and Zhao~\cite{FLZ15} proved $f_4(\delta)=\frac 14+\Theta(\delta)$. 
In an attempt to generalise some of their arguments to larger even cliques we realised that
the values of the Ramsey-Tur\'an density function $f_t(\delta)$ can be determined explicitly 
for $\delta\ll t^{-1}$. Notably in~\cite{LR-a} we proved that
\[
	f_t(\delta)=\begin{cases}
		\frac{t-3}{t-1}+\delta & \text{ if $t\ge 3$ is odd,} \cr
		\frac{3t-10}{3t-4}+\delta-\delta^2 & \text{ if $t\ge 4$ is even}
		\end{cases}
\]
holds provided that $\delta$ is sufficiently small in a sense depending on $t$.
In order to obtain these precise formulae we needed a stability result in the spirit 
of Theorem~\ref{thm:stab} but for the collections~$\ccF_t$ of weighted graphs
mentioned above (see e.g.~\cite{LR-a}*{Proposition 3.5}). While working on this subject, 
we proved analogues of the Andr\'asfai-Erd\H{o}s-S\'os theorem as well.
They form the main contribution of the present work.

\subsection{Results}    

Throughout the rest of this article we only need to deal with weighted graphs~$G=(V, w)$
satisfying $w[V^2]\subseteq \{0, 1, 2\}$. We regard such structures as {\it coloured} complete
graphs on~$V$ by drawing a {\it green}, {\it blue}, or {\it red} edge between any two distinct
vertices $x, y\in V$ depending on whether $w(x, y)$ attains the value $0$, $1$, or $2$.
For a nonnegative integer $n$ the red and blue $n$-vertex clique are denoted by $RK_n$
and $BK_n$, respectively. Moreover, for $n\ge 2$ we mean by $RK_n^-$ the graph obtained 
from an $RK_n$ by recolouring one of its edges blue. 

\begin{dfn} \label{dfn:Gti}
For two integers $a\ge b\ge 1$ the coloured graph $G_{a+b, b}$ of order $a$
consists of an $RK_{b}$ and a $BK_{a-b}$ that are connected to each other by blue edges. 
Moreover, for every positive integer~$t$ we write 
$\ccF_t=\bigl\{G_{t, i}\colon 1\le i\le \frac t2\bigr\}$.
\end{dfn}

\begin{figure}[ht]
\centering
\begin{subfigure}[b]{0.19\textwidth}
    \centering
	\begin{tikzpicture}
	\node[ellipse,fill=white,draw=blue,thick,minimum width=2cm,minimum height=1.33cm, inner sep=0pt] (a) at (0,0) [] { \textcolor{blue}{{\large $BK_{2r-1}$}}};
	\node[ellipse,fill=white,draw=white,minimum width=0cm,minimum height=0cm, inner sep=0pt] (a) at (0,-1.5) [] {};
	\end{tikzpicture}
        \caption{$G_{2r, 1}$}
    \end{subfigure}
\hfill
    \begin{subfigure}[b]{0.19\textwidth}
    \centering
	\begin{tikzpicture}
	\node[ellipse,fill=white,draw=blue,thick,minimum width=2cm,minimum height=1.33cm, inner sep=0pt] (a) at (0,0) [] { \textcolor{blue}{{\large $BK_{2r-4}$}}};
	\node[circle,fill=black,draw,minimum size=0.1cm, inner sep=0pt] (r1) at (1,-2) [] {};
	\node[circle,fill=black,draw,minimum size=0.1cm, inner sep=0pt] (r2) at (-1,-2) [] {};
\begin{scope}[-,>=latex]	
	\foreach \i in {-3,...,3}{		\draw[-, blue] ([xshift=\i * 0.5 cm]r1) -- ([xshift=\i * 0.5 cm]a) ;}
	\foreach \i in {-3,...,3}{		\draw[-, blue] ([xshift=\i * 0.5 cm]r2) -- ([xshift=\i * 0.5 cm]a) ;}
\end{scope}
	\path[red, thick]
		(r1) edge (r2)	
	;
	\end{tikzpicture}
        \caption{$G_{2r, 2}$}
    \end{subfigure}
\hfill
	\begin{subfigure}[b]{0.19\textwidth}
    \centering
		\begin{tikzpicture}
		\node[ellipse,fill=white,draw=red,thick,minimum width=1.5cm,minimum height=1cm, inner sep=0pt] (a) at (0,0) [] { \textcolor{red}{{\large $RK_{i}$}}};	
		\node[ellipse,fill=white,draw=blue,thick,minimum width=1.5cm,minimum height=1cm, inner sep=0pt] (b) at (0,-2) [] { \textcolor{blue}{{\large $BK_{2r-2i}$}}};
		\begin{scope}[-,>=latex]	
			\foreach \i in {-3,...,3}{			\draw[-, blue] ([xshift=\i * 0.5 cm]b) -- ([xshift=\i * 0.5 cm]a) ;}
		\end{scope}
	\end{tikzpicture}
        \caption{$G_{2r, i}$}
    \end{subfigure}
\hfill
    \begin{subfigure}[b]{0.19\textwidth}
    \centering
	\begin{tikzpicture}
	\node[ellipse,fill=white,draw=red,thick,minimum width=2cm,minimum height=1.33cm, inner sep=0pt] (a) at (0,0) [] { \textcolor{red}{{\large $RK_{r-1}$}}};
	\node[circle,fill=black,draw,minimum size=0.1cm, inner sep=0pt] (b1) at (1,-2) [] {};
	\node[circle,fill=black,draw,minimum size=0.1cm, inner sep=0pt] (b2) at (-1,-2) [] {};
\begin{scope}[-,>=latex]	
	\foreach \i in {-3,...,3}{		\draw[-, blue] ([xshift=\i * 0.5 cm]b1) -- ([xshift=\i * 0.5 cm]a) ;}
	\foreach \i in {-3,...,3}{		\draw[-, blue] ([xshift=\i * 0.5 cm]b2) -- ([xshift=\i * 0.5 cm]a) ;}
\end{scope}
	\path[blue, thick]
		(b1) edge (b2)
	;
	\end{tikzpicture}
        \caption{$G_{2r, r-1}$}
    \end{subfigure}
\hfill
    \begin{subfigure}[b]{0.19\textwidth}
    \centering
	\begin{tikzpicture}
	\node[ellipse,fill=white,draw=red,thick,minimum width=2cm,minimum height=1.33cm, inner sep=0pt] (a) at (0,0) [] { \textcolor{red}{{\large $RK_{r}$}}};
	\node[ellipse,fill=white,draw=white,minimum width=0cm,minimum height=0cm, inner sep=0pt] (a) at (0,-1.5) [] {};
	\end{tikzpicture}
        \caption{$G_{2r, r}$}
    \end{subfigure}

    \caption{The family $\ccF_{2r}=\{G_{2r, 1}, G_{2r, 2}, \ldots, G_{2r, r}\}$.}
\end{figure}

Erd\H{o}s, Hajnal, Szemer\'edi, and S\'os proved in~\cite{EHSS} that
\[
	 \pi_{\{0, 1, 2\}}(\ccF_t)=
		\begin{cases}
			\frac{2(t-3)}{t-1} & \text{ if $t\ge 3$ is odd,} \cr
			\frac{2(3t-10)}{3t-4} & \text{ if $t\ge 4$ is even,}
		\end{cases}
\]
which in turn leads to~\eqref{eq:rhot} via the regularity method for graphs. In order to 
state the related results in the spirit of Theorem~\ref{thm:AES} one needs a notion of 
minimum degree for coloured graphs and it will be useful to have a notion of homomorphisms
as well. 

Now if $G=(V, w)$ denotes a coloured graph and $x\in V$, it is natural to call 
\[
	d(x)=\sum_{y\in V}w(x, y)
\]
the {\it degree} of $x$. Moreover, the quantity $\delta(G)=\min\{d(v)\colon v\in V\}$
will be referred to as the {\it minimum degree} of $G$.

\begin{dfn}
A {\it homomorphism} from a weighted graph $G=(V, w)$ to a weighted graph $G'=(V', w')$
is a map $\phi\colon V\lra V'$ with the property that any two distinct vertices 
$x, y\in V$ satisfy $w(x, y)\le w'\bigl(\phi(x), \phi(y)\bigr)$.
\end{dfn} 

For odd indices, we shall obtain the following.
 
\begin{thm} \label{thm:main-odd}
	Suppose that for some $r\ge 2$ we have an $\ccF_{2r+1}$-free coloured graph $G$ 
	of order $n$ with $\delta(G)>\frac{6r-8}{3r-1}n$. Then there is a homomorphism from $G$ 
	to $RK_r$ or, explicitly, there is a partition 
		\[
		V(G)=W_1\dcup\ldots\dcup W_r
	\]
		such that all edges within the partition classes are green.
\end{thm}

Consider the coloured graph obtained from the extremal graph described in 
Subsection~\ref{subsec:simple} by replacing the edges there by red edges 
and colouring all other pairs green. This coloured graph has a minimum degree
of exactly $\frac{6r-8}{3r-1}n$ but, as it does not contain a $BK_{r+1}$, it cannot 
contain a member of $\ccF_{2r+1}$ either. On the other hand, it does not admit
a homomorphism to $RK_r$ and thus it shows that the constant $\frac{6r-8}{3r-1}$
appearing in Theorem~\ref{thm:main-odd} is optimal. Let us also note that the extremal
graphs for $\pi_{\{0, 1, 2\}}(\ccF_{2r+1})=\frac{2(r-1)}{r}$ are $RK_r$ and its symmetric 
blow-ups, which is why we aimed to get a homomorphism into $RK_r$ in the conclusion
of Theorem~\ref{thm:main-odd}. 

As in Ramsey-Tur\'an theory, the even case will be much harder. Let us recall that 
in~\cite{EHSS} the extremal graphs for $\pi_{\{0, 1, 2\}}(\ccF_{2r})=\frac{2(3r-5)}{3r-2}$
have been determined to be certain blow-ups of $RK_r^-$. (E.g., the two ``special'' vertices 
get blown up by a factor of $2$, while the $r-2$ remaining vertices receive a factor of $3$). 
Therefore, our goal is to enforce, by an appropriate minimum degree condition, that an 
$\ccF_{2r}$-free coloured graph admits a homomorphism into~$RK_r^-$.

\begin{thm} \label{thm:main-even}
	Let $r\ge 3$ be an integer and let $G$ be a $\ccF_{2r}$-free coloured graph of order~$n$ 
	with $\delta(G)>\tfrac{14r-24}{7r-5}n$. Then there is a homomorphism from $G$ to $RK_r^-$.
	In other words, there is a partition 
		\[
		V(G)=W_1\dcup\ldots\dcup W_r
	\]
		such that all edges within the partition classes are green and no 
	edge from $W_1$ to $W_2$ is red.
\end{thm}

\begin{figure}[ht]
\centering
\begin{tikzpicture}
	\node[circle,fill=white,draw,minimum size=0.08cm, inner sep=0pt] (a) at (340:3.25) [label=340:$B''$] {$\frac{6n}{7r-5}$};
	\node[circle,fill=white,draw,minimum size=0.08cm, inner sep=0pt] (b) at (200:3.25) [label=240:$B'$] {$\frac{6n}{7r-5}$};
	\node[circle,fill=white,draw,minimum size=0.08cm, inner sep=0pt] (e) at (20:3.25) [label=35: $A_{r-3}$] {$\frac{7n}{7r-5}$};
	\node[circle,fill=white,draw,minimum size=0.08cm, inner sep=0pt] (f) at (160:3.25) [label=145: $A_{1}$] {$\frac{7n}{7r-5}$};
	\node[circle,fill=white,draw,minimum size=0.08cm, inner sep=0pt] (c) at (250:3.25) [label=250:$C'$] {$\frac{2n}{7r-5}$};
	\node[circle,fill=white,draw,minimum size=0.08cm, inner sep=0pt] (d) at (290:3.25) [label=290:$C''$] {$\frac{2n}{7r-5}$};
	\node[circle,fill=black,draw,minimum size=0cm, inner sep=0pt] (v1) at (45:3.25) {};
	\node[circle,fill=black,draw,minimum size=0cm, inner sep=0pt] (v2) at (90:3.25) {};
	\node[circle,fill=black,draw,minimum size=0cm, inner sep=0pt] (v3) at (135:3.25) {};
	\node[circle,fill=white,draw,minimum size=0cm, inner sep=0pt] (w1) at (67.5:3.25) {};
	\node[circle,fill=white,draw,minimum size=0cm, inner sep=0pt] (w2) at (112.5:3.25) {};
	\path[red,thick]
		(a) edge (b)
		(a) edge (c)
		(a) edge (d)
		(a) edge (e)
		(a) edge (f)
		(b) edge (c)
		(b) edge (d)
		(b) edge (e)
		(b) edge (f)
		(c) edge (e)
		(c) edge (f)
		(d) edge (e)
		(d) edge (f)
		(e) edge (f)
		(w1) edge (a)
		(w1) edge (b)
		(w1) edge (c)
		(w1) edge (d)
		(w1) edge (e)
		(w1) edge (f)
		(w2) edge (a)
		(w2) edge (b)
		(w2) edge (c)
		(w2) edge (d)
		(w2) edge (e)
		(w2) edge (f)
		(v1) edge (a)
		(v1) edge (b)
		(v1) edge (c)
		(v1) edge (d)
		(v1) edge (e)
		(v1) edge (f)
		(v2) edge (a)
		(v2) edge (b)
		(v2) edge (c)
		(v2) edge (d)
		(v2) edge (e)
		(v2) edge (f)
		(v3) edge (a)
		(v3) edge (b)
		(v3) edge (c)
		(v3) edge (d)
		(v3) edge (e)
		(v3) edge (f)
		(v1) edge (v2)
		(v1) edge (v3)
		(v1) edge (w1)
		(v1) edge (w2)
		(v2) edge (v3)
		(v2) edge (w1)
		(v2) edge (w2)
		(v3) edge (w1)
		(v3) edge (w2)
		(v2) edge (w2)
		(w1) edge (w2)
	;
	\path[blue,thick]
		(c) edge (b)
		(d) edge (a)
	;
	\path[green,thick]
		(c) edge (d)
	;
\end{tikzpicture}
        \caption{Extremal graph for Theorem~\ref{thm:main-even}.}
        \label{bigJ}
\end{figure}
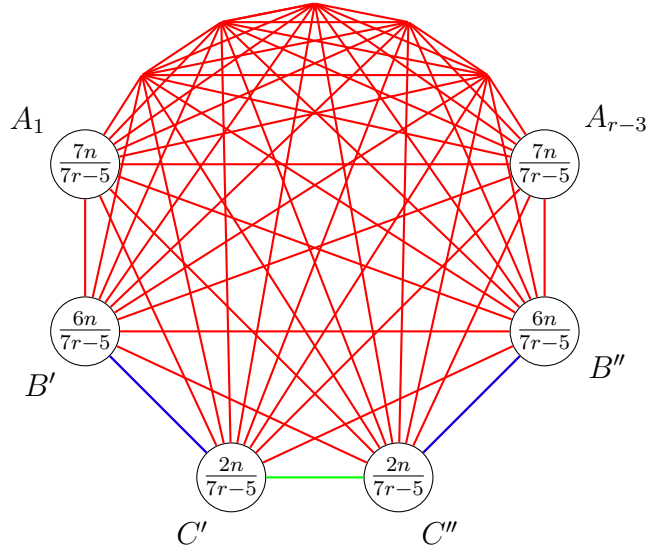

The reason why we stated this only for $r\ge 3$ is that for $r=2$ one can easily show 
a stronger result. This is because $\ccF_2$ consists only of a red edge and a blue triangle.
Hence a direct application of the case $r=2$ of Theorem~\ref{thm:AES} shows that the desired 
conclusion can already be obtained from the weaker minimum degree assumption that 
$\delta(G)>\tfrac 25n$. For $r\ge 3$, however, the constant $\frac{14r-24}{7r-5}$ appearing 
in Theorem~\ref{thm:main-even} is optimal. This can be seen by taking $n$ to be an arbitrary 
multiple of $7r-5$, a vertex set $V$ of size $n$ with a partition
\[
	V=A_1\dcup \ldots \dcup A_{r-3} \dcup B'\dcup B''\dcup C'\dcup C''
\]
satisfying 
\[
	|A_1|=\ldots=|A_{r-3}|=\tfrac{7n}{7r-5}, \quad
	|B'|=|B''|=\tfrac{6n}{7r-5}, \quad 
	\text{ and } \quad |C'|=|C''|=\tfrac{2n}{7r-5}\,,
\]
and colouring  
\begin{enumerate}
	\item[$\bullet$] the edges within the partition classes green,  
	\item[$\bullet$] the edges from $C'$ to $C''$ green as well,
	\item[$\bullet$] the edges from $B'$ to $C'$ and from $B''$ to $C''$ blue, 
	\item[$\bullet$] and all remaining edges red.
\end{enumerate}

We would like to remark that whenever a weighted Tur\'an density $\pi_D(\ccF)$
and the corresponding family $\mathscr{E}$ of extremal graphs have been determined 
one may ask, similarly, for the {\it Andr\'asfai-Erd\H{o}s-S\'os threshold} $\alpha_D(\ccF)$, 
defined to be the infimal real number $\alpha$ with the following property: 
Every $\ccF$-free weighted graph $(V, w)$ with $w[V^2]\subseteq D$ and $\delta(G) > \alpha |V|$ admits an homomorphism into a member of $\mathscr{E}$. For instance, Theorem~\ref{thm:main-even}
and the graph in Figure~\ref{bigJ} show $\alpha_{\{0, 1, 2\}}(\ccF_{2r})=\frac{14r-24}{7r-5}$
for $r\ge 3$. It would be interesting to study such thresholds $\alpha_D(\ccF)$ in  
further cases, e.g. for the pairs $(D, \ccF)$ whose Tur\'an densities have been determined
in~\cite{FK02}.

\section{Excluding blue cliques}

Many intermediate steps in the proofs of our main results are of the 
following form: We already know that the coloured graph $G$ under consideration 
is $\ccF$-free for some set $\ccF$ of coloured graphs  
and we would like to show that for a certain other coloured graph~$F$ 
it must be the case that~$G$ is $F$-free as well. The usual strategy for handling 
such a problem begins by assigning a positive integral {\it weight} $\gamma_z$ to 
every $z\in V(F)$. Assuming for simplicity 
that~$F$ itself would be a subgraph of~$G$ we obtain
\[
	\sum_{x\in V}\sum_{z\in V(F)} \gamma_z w(x, z)
	=\sum_{z\in V(F)}\gamma_z d_G(z)
	\ge \delta(G) \sum_{z\in V(F)}\gamma_z\,.
\]
Consequently there will exist some vertex $x\in V$ such that
\begin{align}\label{eq:xnew}
	\sum_{z\in V(F)} \gamma_z w(x, z) \ge \frac{\delta(G)}{n} \sum_{z\in V(F)}\gamma_z\,.
\end{align}

The basic plan to proceed from this point is that we try to prove by means of some case analysis
that this conditions implies $V(F)\cup\{x\}$ to support some member of~$\ccF$, contrary 
to $G$ being $\ccF$-free. 
Often the situation will be a bit more complicated and we will need to iterate this argument 
multiple times before such a contradiction emerges. For analysing the condition~\eqref{eq:xnew} 
it is usually helpful to rewrite it in terms of the function
$\wq\colon V^2\lra \{0, 1, 2\}$ defined by $\wq(u, v)=2-w(u, v)$ for all $u, v\in V$.
In fact one can easily check that~\eqref{eq:xnew} is equivalent to 
\begin{align}\label{eq:xnew+}
	  \sum_{z\in V(F)} \gamma_z \wq(x, z) 
	  \le \left(2-\frac{\delta(G)}{n}\right) \sum_{z\in V(F)}\gamma_z\,.
\end{align}

Our first argument of this form will establish the following lemma, which will later 
be used to show that a coloured graph $G$ satisfying the assumption of either 
Theorem~\ref{thm:main-odd} or Theorem~\ref{thm:main-even} cannot contain a $BK_{r+1}$ 
(see Lemma~\ref{lem:bkr1} and Lemma~\ref{lem:bkr2} below).

\begin{lemma}\label{lem:blue}
	Let $q > b\ge 1$ be integers and suppose that $G$ is a coloured 
	graph on~$n$ vertices with $\delta(G)>\bigl(2-\frac{12}{3q+3b-5}\bigr)n$
	containing a $BK_q$.
	Then either $BK_{q+1}$ or $G_{q+b, b}$ is a subgraph of $G$.
\end{lemma}

\begin{proof}
	Assume contrariwise that $G$ is $\{BK_{q+1}, G_{q+b, b}\}$-free. 
	For each integer $k\in [0, b-1]$ we define 
		\[
		p_k=\max(0, k+q+1-2b)\,.
	\]
	In view of
		\begin{align}\label{eq:Cnonempty}
		k+p_k=\max(k, 2k+q+1-2b) \le \max(b-1, q-1)< q
	\end{align}
	there exists a coloured graph $H_k$ of order $q$ having a vertex partition 
	$V(H_k)=A\dcup B\dcup C$ satisfying
	\begin{enumerate}
	\item[$\bullet$] $|A|=k$, $|B|=p_k$, $|C|=q-(k+p_k)$,
	\item[$\bullet$] all edges of $H_k$ connecting a vertex in $A$ with a vertex in 
		$A\cup C$ are red,
	\item[$\bullet$] and all other edges of $H_k$ are blue.
	\end{enumerate}

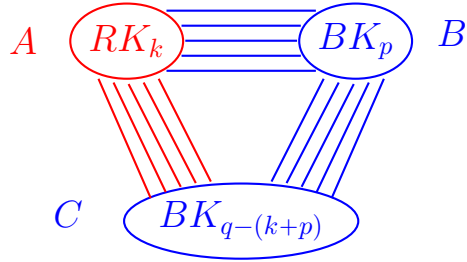
\begin{figure}[ht]
\centering
\begin{tikzpicture}
	\node[ellipse ,fill=white,draw=red,thick,minimum width=1.5cm,minimum height=1cm, inner sep=0pt][] (a) at (162:1.6) { \textcolor{red}{{\large $RK_{k}$}}};	
	\node[ellipse,fill=white,draw=blue,thick,minimum width=1.5cm,minimum height=1cm, inner sep=0pt][] (b) at (18:1.6) { \textcolor{blue}{{\large $BK_{p}$}}};
	\node[ellipse,fill=white,draw=blue,thick,minimum width=2cm,minimum height=1cm, inner sep=0pt][] (c) at (-90:1.9) { \textcolor{blue}{{\large $BK_{q-(k+p)}$}}};
	\node at (-2.9,0.5) {\textcolor{red}{\large $A$}};
	\node at (2.8,0.6) {\textcolor{blue}{\large $B$}};
	\node at (-2.3,-1.8) {\textcolor{blue}{\large $C$}};
	\node[circle,draw=none,fill=none,minimum size=0cm,inner sep=0pt] (a1) at (-1.1,0) {};
	\node[circle,draw=none,fill=none,minimum size=0cm,inner sep=0pt] (a2) at (-1.3,-0.05) {};
	\node[circle,draw=none,fill=none,minimum size=0cm,inner sep=0pt] (a3) at (-1.5,-0.07) {};
	\node[circle,draw=none,fill=none,minimum size=0cm,inner sep=0pt] (a4) at (-1.7,-0.05) {};
	\node[circle,draw=none,fill=none,minimum size=0cm,inner sep=0pt] (a5) at (-1.9,0) {};
	\node[circle,draw=none,fill=none,minimum size=0cm,inner sep=0pt] (a6) at (-1,0.9) {};
	\node[circle,draw=none,fill=none,minimum size=0cm,inner sep=0pt] (a7) at (-0.8,0.7) {};
	\node[circle,draw=none,fill=none,minimum size=0cm,inner sep=0pt] (a8) at (-0.75,0.5) {};
	\node[circle,draw=none,fill=none,minimum size=0cm,inner sep=0pt] (a9) at (-0.8,0.3) {};
	\node[circle,draw=none,fill=none,minimum size=0cm,inner sep=0pt] (a10) at (-1,0.1) {};
	\node[circle,draw=none,fill=none,minimum size=0cm,inner sep=0pt] (b1) at (1,0.9) {};
	\node[circle,draw=none,fill=none,minimum size=0cm,inner sep=0pt] (b2) at (0.8,0.7) {};
	\node[circle,draw=none,fill=none,minimum size=0cm,inner sep=0pt] (b3) at (0.75,0.5) {};
	\node[circle,draw=none,fill=none,minimum size=0cm,inner sep=0pt] (b4) at (0.8,0.3) {};
	\node[circle,draw=none,fill=none,minimum size=0cm,inner sep=0pt] (b5) at (1,0.1) {};
	\node[circle,draw=none,fill=none,minimum size=0cm,inner sep=0pt] (b6) at (1.1,0) {};
	\node[circle,draw=none,fill=none,minimum size=0cm,inner sep=0pt] (b7) at (1.3,-0.05) {};
	\node[circle,draw=none,fill=none,minimum size=0cm,inner sep=0pt] (b8) at (1.5,-0.07) {};
	\node[circle,draw=none,fill=none,minimum size=0cm,inner sep=0pt] (b9) at (1.7,-0.05) {};
	\node[circle,draw=none,fill=none,minimum size=0cm,inner sep=0pt] (b10) at (1.9,0) {};
	\node[circle,draw=none,fill=none,minimum size=0cm,inner sep=0pt] (c1) at (-0.4,-1.38) {};
	\node[circle,draw=none,fill=none,minimum size=0cm,inner sep=0pt] (c2) at (-0.6,-1.4) {};
	\node[circle,draw=none,fill=none,minimum size=0cm,inner sep=0pt] (c3) at (-0.8,-1.45) {};
	\node[circle,draw=none,fill=none,minimum size=0cm,inner sep=0pt] (c4) at (-1,-1.52) {};
	\node[circle,draw=none,fill=none,minimum size=0cm,inner sep=0pt] (c5) at (-1.2,-1.6) {};
	\node[circle,draw=none,fill=none,minimum size=0cm,inner sep=0pt] (c6) at (0.4,-1.38) {};
	\node[circle,draw=none,fill=none,minimum size=0cm,inner sep=0pt] (c7) at (0.6,-1.4) {};
	\node[circle,draw=none,fill=none,minimum size=0cm,inner sep=0pt] (c8) at (0.8,-1.45) {};
	\node[circle,draw=none,fill=none,minimum size=0cm,inner sep=0pt] (c9) at (1,-1.52) {};
	\node[circle,draw=none,fill=none,minimum size=0cm,inner sep=0pt] (c10) at (1.2,-1.6) {};
	\path[red,thick]
		(a1) edge (c1)
		(a2) edge (c2)
		(a3) edge (c3)
		(a4) edge (c4)
		(a5) edge (c5)
	;
	\path[blue,thick]
		(a6) edge (b1)
		(a7) edge (b2)
		(a8) edge (b3)
		(a9) edge (b4)
		(a10) edge (b5)
		(c6) edge (b6)
		(c7) edge (b7)
		(c8) edge (b8)
		(c9) edge (b9)
		(c10) edge (b10)
	;
\end{tikzpicture}
\caption{The coloured graph $H_k$}
\end{figure}

	Since all edges of $H_0$ are blue and $V(H_0)$ has size $q$, we know that $G$ contains 
	a copy of~$H_0$. Now let $k_*$ denote the largest integer in $[0, b-1]$ with the 
	property that~$G$ contains a copy of $H_{k_*}$ and put $p_*=p_{k_*}$. 

	Let $A\dcup B\dcup C\subseteq V(G)$ be the vertex set of such a copy with the notation 
	as above. Notice that the calculation~\eqref{eq:Cnonempty} shows $C\ne\varnothing$. 
	So if $k_*=b-1$, then $A$ and an arbitrary vertex in $C$ would form an $RK_{b}$, 
	while the remaining vertices in $B\cup C$ would form a~$BK_{q-b}$. 
	Due to the absence of green edges from $A$ to $B\dcup C$ this means that $G$ 
	would contain a $G_{q+b, b}$, which is absurd. 
	
	This consideration proves  
		\begin{equation}\label{eq:ksmall}
		k_*\le b-2
	\end{equation} 
		and our maximal choice of $k_*$ entails that $G$ does not contain an $H_{k_*+1}$.

	\smallskip

	{\it \hskip1em First Case: $k_*<2b-q-1$.} 

	\medskip

	This yields $p_*=p_{k_*+1}=0$ and $B=\varnothing$. We assign the weight~$3$ to the 
	vertices in~$A$ and the weight~$2$ to the vertices in~$C$. In view of $q>b$ the total 
	weight of all vertices in $A\dcup C$ is
		\[
		3k_*+2(q-k_*)=2q+k_*\le 2b+q-2 \le 2b+q-2+\tfrac12 (q-b-1)=\tfrac 12(3q+3b-5)\,.
	\]
		Writing $\gamma_z$ for the weight of every vertex $z\in A\dcup C$ we find, 
	by the argument leading to~\eqref{eq:xnew+}, a vertex $x\in V(G)$ satisfying 
		\[
		\sum_{z\in A\cup C} \gamma_z\wq(x, z)<\frac{6(3q+3b-5)}{3q+3b-5}=6\,.
	\]
		Owing to the integrality of the left side we obtain
		\begin{equation}\label{eq:x1}
		3\sum_{a\in A} \wq(x, a) + 2\sum_{c\in C} \wq(x, c) \le 5\,.
	\end{equation}
		Because of $BK_{q+1}\not\subseteq G$ there exists a vertex $z\in A\cup C$
	with $\wq(x, z)=2$. In view of~\eqref{eq:x1} this can only happen if $z\in C$ 
	and thus we infer
		\[
		3\sum_{a\in A} \wq(x, a) + 2\sum_{c\,\in\, C\setminus \{z\}} \wq(x, c) \le 1\,,
	\]
		which in turn tells us that all members of $A\cup C\setminus \{z\}$ are red neighbours 
	of $x$. Moreover,~${z\in C}$ and $\wq(x, z)=2$ imply $x\not\in A$. 
	Consequently, $(A\cup\{x\})\dcup (C\setminus \{z\})$ is the vertex partition 
	of an $H_{k_*+1}$ in $G$ and we have reached a contradiction.

	\smallskip

	{\it \hskip1em Second Case: $k_*\ge 2b-q-1$.} 

	\medskip

	Observe that now we have $p_*=k_*+q+1-2b$ and $p_{k_*+1}=p_*+1$. 
	This time we assign the weight $\gamma_z=2$ 
	to every $z\in A$ and the weight $\gamma_z=1$ to every $z\in B\dcup C$. As before we find a 
	vertex $x\in V(G)$ with 
		\[
		\sum_{z\in A\cup B\cup C} \gamma_z\wq(x, z)<\frac{12(k_*+q)}{3q+3b-5}
		\,\overset{\text{\eqref{eq:ksmall}}}{<}\,
		\frac{12(q+b-2)}{3(q+b-2)} = 4\,,
	\]
		i.e.,
		\[
		2\sum_{a\in A} \wq(x, a) + \sum_{y\in B\cup C} \wq(x, y) \le 3\,.
	\]
		Again the absence of a $BK_{q+1}$ in $G$ leads us to a vertex $z\in B\cup C$ with 
	$\wq(x, z)=2$. Moreover, there are only red edges from $x$ to $A$ and at most one blue 
	but no green edges from $x$ to $B\cup C\setminus\{z\}$. This implies, however, that 
	$(A\cup \{x\})\cup(B\cup C\setminus\{z\})$ supports an $H_{k_*+1}$ in~$G$, which is again a 
	contradiction.
\end{proof}

\section{The proof of Theorem~\ref{thm:main-odd}}

An iterative application of Lemma~\ref{lem:blue} leads to the following result. 

\begin{lemma}\label{lem:bkr1}
	For $r\ge 2$ every $\ccF_{2r+1}$-free coloured graph $G$ of order~$n$
	with $\delta(G)>\tfrac{6r-8}{3r-1}n$ is $BK_{r+1}$-free.
\end{lemma}  

\begin{proof}
	Let $q$ be maximal with $BK_q\subseteq G$ and assume for the sake of 
	contradiction that $q\ge r+1$. Due to $G_{2r+1, 1}=BK_{2r}$ we have
	$q < 2r$ and, hence, the number $b=2r+1-q$ satisfies $q>b\ge 1$.
	Since $\delta(G)>\bigl(2-\frac{12}{3(2r+1)-5}\bigr)n$, it follows from 
	Lemma~\ref{lem:blue} that either $BK_{q+1}\subseteq G$ or $G_{2r+1, b}\subseteq G$.
	The former, however, contradicts the maximality of $q$ and the latter contradicts
	$G$ being $\ccF_{2r+1}$-free.
\end{proof}

Now Theorem~\ref{thm:main-odd} follows by means of a simple application of the 
Andr\'asfai-Erd\H{o}s-S\'os theorem.

\begin{proof}[Proof of Theorem~\ref{thm:main-odd}]
	Let $H$ denote the simple graph on $V(G)$ whose edges correspond to the blue 
	or red edges of $G$. The minimum degree condition on $G$ yields 
		\[
		\delta(H)\ge \tfrac12\delta(G)>\tfrac{3r-4}{3r-1}n
	\]
		and Lemma~\ref{lem:bkr1} tells us that~$H$ is $K_{r+1}$-free. 
	So by Theorem~\ref{thm:AES} $H$ is $r$-partite and the claim follows.
\end{proof}

\section{The proof of Theorem~\ref{thm:main-even}}

Again we begin by utilising Lemma~\ref{lem:blue}.

\begin{lemma}\label{lem:bkr2}
	For $r\ge 3$ every $\ccF_{2r}$-free coloured graph $G$ of order~$n$
	with $\delta(G)>\tfrac{14r-24}{7r-5}n$ is $BK_{r+1}$-free.
\end{lemma}  

\begin{proof}
	As in the proof of Lemma~\ref{lem:bkr1} we look at the largest integer $q$
	with $BK_q\subseteq G$ and observe that $G_{2r, 1}=BK_{2r-1}$ shows $q\le 2r-1$.
	So assuming $q\ge r+1$ Lemma~\ref{lem:blue} would again tell us that either
	$BK_{q+1}$ or $G_{2r, 2r-q}$ is a subgraph of $G$, both of which is absurd.
	Actually, this argument only requires the lower bound 
	$\delta(G)>\bigl(2-\frac{12}{6r-5}\bigr)n$ on the minimum degree of $G$, 
	which is less than what we stated.
\end{proof}

In order to define the homomorphism demanded by Theorem~\ref{thm:main-even} it would be
tremendously helpful to know that no induced subgraph of $G$ with three vertices has
exactly one red edge. While not being true in general, this assertion will turn out 
to hold in the important special case that all edges of $G$ that are not themselves red 
belong to the common red neigbourhood of some $RK_{r-2}$ (see Lemma~\ref{lem:wicked} below).
This property of $G$ can in turn be derived from a certain
``edge-maximality'' condition (see Lemma~\ref{lem:blue-secure} and Lemma~\ref{lem:green-secure} 
below). The definition that follows facilitates talking about this plan.

\begin{dfn}
	Let $G=(V, w)$ be a coloured graph.
	\begin{enumerate}[label=\alabel]
	\item  If $G$ is $\ccF_{2r}$-free and every $\ccF_{2r}$-free coloured graph $G'=(V, w')$ 
	having $G$ as a subgraph (i.e., satisfying $w'(x, y)\ge w(x, y)$ for all $x, y\in V$)
	coincides with $G$, then we say that $G$ is {\it extremal}.
	\item A blue or green edge of $G$ is called {\it secure} if it is contained 
	in the common red neighbourhood of some $RK_{r-2}$.
	\item A {\it wicked triangle} in $G$ is a triple $(x, y, z)$ of distinct vertices, such that 
	$xy$ is red and $xz$, $yz$ are either blue or green. 
	If in this situation both $xz$ and $yz$ are blue, then $(x, y, z)$ is said to be 
	a {\it blue wicked triangle}. 
	\end{enumerate}
\end{dfn}

We shall see later that one only needs to deal with the extremal case when proving 
Theorem~\ref{thm:main-even}. As indicated above we will prove in this case that 
all blue and green edges are indeed secure and that wicked triangles do not exist.
We commence with the easiest of these claims, the security of blue edges.

\begin{lemma}\label{lem:blue-secure}
	If $G$ designates an extremal $\ccF_{2r}$-free coloured graph with $n$ vertices and 
	$\delta(G)>\tfrac{14r-24}{7r-5}n$, then all blue edges of $G$ are secure.
\end{lemma}

\begin{proof}
Let $xy$ denote an arbitrary blue edge of $G$. By extremality, the weighted graph~$G'$
arising from $G$ by recolouring $xy$ red contains, for some $i\in [r]$, a subgraph isomorphic
to $G_{2r, i}$. If $i\ne r$ this subgraph would have at least $r+1$ vertices no two of which are 
connected by a green edge in $G$. Consequently, $G$ would contain a $BK_{r+1}$, which contradicts
Lemma~\ref{lem:bkr2}. So $G'$ contains an $RK_r$ and, as $G$ was $RK_r$-free, the vertices 
$x$ and $y$ must belong to this $RK_r$. Its remaining $r-2$ vertices form, in $G$, an $RK_{r-2}$
whose red neighbourhood contains $x$ and $y$. 
\end{proof}

At this moment we could already rule out the existence of blue wicked triangles 
(see part~\ref{it:wl1} of Lemma~\ref{lem:wicked} below). However, the argument
for doing so is very similar to the proof that, provided the green edges are secure
as well, there cannot be any wicked triangles at all. For this reason we postpone this 
step and consider the green edges first. But it will be important to remember
that we may already assume the absence of blue wicked triangles when treating the security
of green edges. 

As a further preparation towards this latter task we need to exclude a configuration
that is closely tied to the example given at the end of the introduction demonstrating the optimality of the minimum degree 
condition in Theorem~\ref{thm:main-even}.

\begin{dfn}\label{dfn:J}
	By $J$ we mean the coloured graph of order $r+1$ with vertex set 
		\[
		A\dcup\{b', b'', c', c''\}\,,
	\]
		where $|A|=r-3$, $c'c''$ is green, and $b'c'$, $b''c''$ are blue,
	while all other edges are red.
\end{dfn}

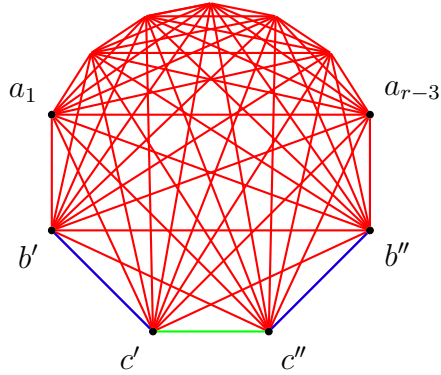
\begin{figure}[ht]
	\centering
	\begin{tikzpicture}[scale=1.5]
	\node[circle,fill=black,draw,minimum size=0.09cm, inner sep=0pt] (a) at (340:1.5) [label=340:$b''$] {};
	\node[circle,fill=black,draw,minimum size=0.09cm, inner sep=0pt] (b) at (200:1.5) [label=240:$b'$] {};
	\node[circle,fill=black,draw,minimum size=0.09cm, inner sep=0pt] (e) at (20:1.5) [label=35: $a_{r-3}$] {};
	\node[circle,fill=black,draw,minimum size=0.09cm, inner sep=0pt] (f) at (160:1.5) [label=145: $a_{1}$] {};
	\node[circle,fill=black,draw,minimum size=0.09cm, inner sep=0pt] (c) at (250:1.5) [label=250:$c'$] {};
	\node[circle,fill=black,draw,minimum size=0.09cm, inner sep=0pt] (d) at (290:1.5) [label=290:$c''$] {};
	\node[circle,fill=black,draw,minimum size=0cm, inner sep=0pt] (v1) at (45:1.5) {};
	\node[circle,fill=black,draw,minimum size=0cm, inner sep=0pt] (v2) at (90:1.5) {};
	\node[circle,fill=black,draw,minimum size=0cm, inner sep=0pt] (v3) at (135:1.5) {};
	\node[circle,fill=white,draw,minimum size=0cm, inner sep=0pt] (w1) at (67.5:1.5) {};
	\node[circle,fill=white,draw,minimum size=0cm, inner sep=0pt] (w2) at (112.5:1.5) {};
	\path[red,thick]
		(a) edge (b)
		(a) edge (c)
		(a) edge (d)
		(a) edge (e)
		(a) edge (f)
		(b) edge (c)
		(b) edge (d)
		(b) edge (e)
		(b) edge (f)
		(c) edge (e)
		(c) edge (f)
		(d) edge (e)
		(d) edge (f)
		(e) edge (f)
		(w1) edge (a)
		(w1) edge (b)
		(w1) edge (c)
		(w1) edge (d)
		(w1) edge (e)
		(w1) edge (f)
		(w2) edge (a)
		(w2) edge (b)
		(w2) edge (c)
		(w2) edge (d)
		(w2) edge (e)
		(w2) edge (f)
		(v1) edge (a)
		(v1) edge (b)
		(v1) edge (c)
		(v1) edge (d)
		(v1) edge (e)
		(v1) edge (f)
		(v2) edge (a)
		(v2) edge (b)
		(v2) edge (c)
		(v2) edge (d)
		(v2) edge (e)
		(v2) edge (f)
		(v3) edge (a)
		(v3) edge (b)
		(v3) edge (c)
		(v3) edge (d)
		(v3) edge (e)
		(v3) edge (f)
		(v1) edge (v2)
		(v1) edge (v3)
		(v1) edge (w1)
		(v1) edge (w2)
		(v2) edge (v3)
		(v2) edge (w1)
		(v2) edge (w2)
		(v3) edge (w1)
		(v3) edge (w2)
		(v2) edge (w2)
		(w1) edge (w2)
	;
	\path[blue,thick]
		(c) edge (b)
		(d) edge (a)
	;
	\path[green,thick]
		(c) edge (d)
	;
\end{tikzpicture}
        \captionof{figure}{The coloured graph $J$ with $A=\{a_1, \ldots, a_{r-3}\}$.}\label{GGr}
\end{figure}

\begin{lemma}\label{lem:exclude-J}
A $\{RK_r, BK_{r+1}\}$-free coloured graph $G$ of order $n$ with 
$\delta(G)>\tfrac{14r-24}{7r-5}n$ cannot contain $J$ as subgraph.
\end{lemma}
 
\begin{proof}
Otherwise let $Q=A\dcup\{b', b'', c', c''\}\subseteq V(G)$ be the vertex set of a copy of $J$
in $G$, the notation being as in Definition~\ref{dfn:J}.  
We assign weights to the members of $Q$ according to the formula
\[
	\gamma_q=
	\begin{cases}
		7  & \text{ if } q\in A, \cr
		6  & \text{ if } q\in \{b', b''\},  \cr
		2  & \text{ if } q\in \{c', c''\}.                      
	\end{cases}
\]
So the total weight of all vertices is $7(r-3)+6\cdot 2+2\cdot 2=7r-5$ and the standard 
argument leads to a vertex $x\in V(G)$ with
\begin{align}\label{eq:J1}
	\sum_{q\in Q}\gamma_q\wq(x, q)\le 13\,.
\end{align}

This inequality allows us to analyse the set $T=\{q\in Q\colon \wq(x, q)=2\}$. As as immediate 
consequence of~\eqref{eq:J1} we have $T\subseteq\{b', b'', c', c''\}$. Moreover, the assumption
$b'\in T$ would imply that $Q\setminus\{b'\}$ is contained in the red neighbourhood of $x$,
but then $A\cup\{b'', c', x\}$ would induce an $RK_r$ in $G$, which is absurd. By symmetry 
the same consideration applies to~$b''$ as well and thus we have $T\subseteq\{c', c''\}$.

Now it follows from $A\cup\{b', b'', c', x\}$ not spanning a $BK_{r+1}$ in $G$  
that $c'\in T$ and, similarly, we get $c''\in T$ as well. By plugging~$T=\{c', c''\}$
into~\eqref{eq:J1} we learn 
\[
	   6\sum_{q\in A\cup\{b', b''\}}\wq(x, q)\le 5\,, 
\]
and for this reason $A\cup\{b', b''\}$ is part of the red neighbourhood of $x$.
But this means that~$A\cup\{b', b'', x\}$ forms an $RK_r$ in $G$, which is absurd.
\end{proof}

Now we proceed with the security of green edges.
The argument starts in a similar way as the proof of Lemma~\ref{lem:blue-secure},
but there will be more cases to investigate.

\begin{lemma}\label{lem:green-secure}
Let $G$ be an extremal $\ccF_{2r}$-free coloured graph with $\delta(G)>\tfrac{14r-24}{7r-5}n$.
If $G$ contains no blue wicked triangle, then all green edges of $G$ are secure.
\end{lemma}

\begin{proof}
Recall that $G$ has to be $\{BK_{r+1}, J\}$-free by Lemma~\ref{lem:bkr2} and 
Lemma~\ref{lem:exclude-J}. 
Now consider any green edge $xy$ of $G$ and denote the coloured graph that one obtains from~$G$ 
when one recolours $xy$ to become blue by $G'$. Due to the extremality of $G$ we know that 
$G'$ cannot be $\ccF_{2r}$-free. Exploiting that $G$ is $BK_{r+1}$-free it is easily seen that
$G'$ must contain a~$G_{2r, r-1}$ with~$x$ and $y$ among its vertices. This $G_{2r, r-1}$ is,
of course, only known to be a subgraph of $G'$ that does not need to be induced. In fact,
the absence of blue wicked triangles in~$G$ entails that ``many'' edges of this subgraph that
``in general'' would only be known to be either blue or red must actually be red. 
To get an overview over the possible cases,  
we observe that due to the symmetry between $x$ and $y$ one may assume that for
the ``distinguished'' blue edge of the $G_{2r, r-1}$ one of the following three 
cases occurs.
\begin{enumerate}[label=\alabel]
\item\label{it:edge-a} It is $xy$.
\item\label{it:edge-b} It is of the form $xb$ and genuinely blue, where $b$ is in the $RK_{r-1}$. 
\item\label{it:edge-c} It is of the form $xc$ and red, where $c$ is in the $RK_{r-1}$.
\end{enumerate}

In case~\ref{it:edge-a} there may be at most one blue edge $xa_x$ from $x$ into the $RK_{r-1}$,
since otherwise $G$ would contain a blue wicked triangle. 
For the same reason, there can be at most one blue edge $ya_y$ from $y$ into the $RK_{r-1}$. 
If both blue edges exist, then $J\not\subseteq G$ implies $a_x=a_y$ and we get the 
configuration shown in Figure~\ref{fig:xy1}. Similarly, the above cases~\ref{it:edge-b} 
and~\ref{it:edge-c} lead to one of the situations in Figure~\ref{fig:edgexy}. 
Observe that it might still be the case that some of the edges drawn blue in these pictures are 
actually red in $G$. 

\begin{figure}[ht]
\centering
    \begin{subfigure}[b]{0.32\textwidth}
    \centering
	\begin{tikzpicture}
	\node[ellipse,fill=white,draw=red,thick,minimum width=1.5cm,minimum height=1cm, inner sep=0pt] (z) at (135:1.25) [] { \textcolor{red}{{\large $RK_{r-2}$}}};
	\node[circle,fill=white,draw,minimum size=0.5cm, inner sep=0pt] (a) at (45:1.25) [] {$a$};
	\node[circle,fill=white,draw,minimum size=0.5cm, inner sep=0pt] (c) at (225:1.25) [] {$x$};
	\node[circle,fill=white,draw,minimum size=0.5cm, inner sep=0pt] (d) at (315:1.25) [] {$y$};
	\path[red,thick]
		(a) edge (z)
		(c) edge (z)
		(d) edge (z)
		(a) edge (d)
			;
	\path[blue,thick]
		(a) edge (d)
		(a) edge (c)		
	;
	\path[green,thick]
		(c) edge (d)		
	;
	\end{tikzpicture}
        \caption{First case}\label{fig:xy1}
    \end{subfigure}
\hfill    
\begin{subfigure}[b]{0.32\textwidth}
    \centering
	\begin{tikzpicture}
	\node[ellipse,fill=white,draw=red,thick,minimum width=1.5cm,minimum height=1cm, inner sep=0pt] (z) at (90:1.25) [] { \textcolor{red}{{\large $RK_{r-3}$}}};
	\node[circle,fill=white,draw,minimum size=0.5cm, inner sep=0pt] (a) at (180:1.25) [] {$a$};
	\node[circle,fill=white,draw,minimum size=0.5cm, inner sep=0pt] (b) at (0:1.25) [] {$b$};
	\node[circle,fill=white,draw,minimum size=0.5cm, inner sep=0pt] (c) at (300:1.25) [] {$x$};
	\node[circle,fill=white,draw,minimum size=0.5cm, inner sep=0pt] (d) at (240:1.25) [] {$y$};
	\path[red,thick]
		(a) edge (z)
		(b) edge (z)
		(c) edge (z)
		(d) edge (z)
		(b) edge (d)
		(a) edge (d)
	;
	\path[blue,thick]
		(a) edge (b)
		(a) edge (c)
		(b) edge (c)		
	;
	\path[green,thick]
		(c) edge (d)		
	;
	\end{tikzpicture}
        \caption{Second case}\label{fig:xy2}
    \end{subfigure}
\hfill
       \begin{subfigure}[b]{0.32\textwidth}
    \centering
	\begin{tikzpicture}
	\node[ellipse,fill=white,draw=red,thick,minimum width=1.5cm,minimum height=1cm, inner sep=0pt] (z) at (120:1.5) [] { \textcolor{red}{{\large $RK_{r-4}$}}};
	\node[circle,fill=white,draw,minimum size=0.5cm, inner sep=0pt] (a) at (180:1.25) [] {$a$};
	\node[circle,fill=white,draw,minimum size=0.5cm, inner sep=0pt] (b) at (0:1.25) [] {$b$};
	\node[circle,fill=white,draw,minimum size=0.5cm, inner sep=0pt] (c) at (300:1.25) [] {$x$};
	\node[circle,fill=white,draw,minimum size=0.5cm, inner sep=0pt] (d) at (240:1.25) [] {$y$};
	\node[circle,fill=white,draw,minimum size=0.5cm, inner sep=0pt] (e) at (60:1.25) [] {$c$};
	\path[red,thick]
		(a) edge (z)
		(b) edge (z)
		(c) edge (z)
		(d) edge (z)
		(e) edge (z)
		(b) edge (d)
		(a) edge (d)
		(a) edge (c)
		(c) edge (e)
		(d) edge (e)
		(a) edge (b)
		(b) edge (e)
	;
	\path[blue,thick]
		(a) edge (e)
		(b) edge (c)
	;
	\path[green,thick]
		(c) edge (d)		
	;
	\end{tikzpicture}
        \caption{Third case}\label{fig:xy3}
    \end{subfigure}    
    \caption{Possibilities for the edge $xy$.}\label{fig:edgexy}
    \vspace{-1em}
\end{figure}
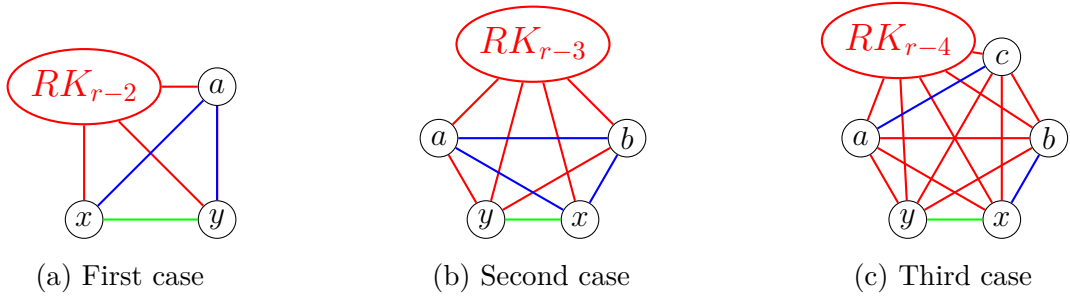

From now on we treat these three cases separately. 
If the configuration depicted in Figure~\ref{fig:xy1} 
occurs the edge $xy$ is secure due to the $RK_{r-2}$ shown there.

Suppose next that we are in the case shown in Figure~\ref{fig:xy2} and let $Q$ be the 
vertex set of the $RK_{r-3}$. Assign 
\begin{enumerate}
	\item[$\bullet$] the weight $1$ to $a$, $b$, $x$, $y$,
	\item[$\bullet$] and the weight $2$ to the members of $Q$.
\end{enumerate}
So the total weight is $2(r-1)$ and thus there is a vertex $v$ with
\[
	\sum_{z\in\{a, b, x, y\}}\wq(v, z)+2\sum_{q\in Q} \wq(v, q)\le 3\,.
\]

In combination with neither $Q\cup \{a, b, v, x\}$ nor $Q\cup \{a, b, v, y\}$
forming a $BK_{r+1}$ this implies that either $\wq(a, v)=2$ or $\wq(b, v)=2$.
By symmetry we may suppose that the latter holds, thus getting
\[
	\sum_{z\in\{a, x, y\}}\wq(v, z) +2\sum_{q\in Q} \wq(v, q)\le 1\,.
\]

It follows that all vertices in $Q$ and at least two of $a$, $x$, and $y$ are red 
neighbours of $v$. Moreover, $v\not\in\{a, x, y\}$ and none of the edges $va$, $vx$, 
and $vy$ is green. Now if $va$ and $vy$ are red, then $Q\cup\{a, v, y\}$ forms
an $RK_{r}$ in $G$, which is absurd. Furthermore, if $va$ and $vx$ are red,  
then $Q\cup\{a, v, x, y\}$ forms a copy of $J$ in $G$, which is not possible either.
So the only remaining case is that $vx$ and $vy$ are red and then $Q\cup\{v\}$ 
forms an~$RK_{r-2}$ exemplifying the security of $xy$. 

It remains to discuss the configuration shown in Figure~\ref{fig:xy3}, 
which can only arise if~$r\ge 4$. 
This time we let $Q$ denote the vertex set of the $RK_{r-4}$. Assigning 
\begin{enumerate}
\item[$\bullet$] the weight $4$ to $x$, $y$,
\item[$\bullet$] the weight $5$ to $a$, $b$, $c$,
\item[$\bullet$]  and the weight $7$ to the members of $Q$
\end{enumerate}
we have distributed a total weight of $7r-5$ and in the usual manner we find a vertex $v$
with
\[
	  4\sum_{z\in\{x, y\}}\wq(v, z)+5\sum_{z\in\{a, b, c\}}\wq(v, z)+
	  7\sum_{q\in Q} \wq(v, q)\le 13\,.
\]

Exploiting that neither $Q\cup\{a, b, c, x\}$ nor $Q\cup\{a, b, c, y\}$ induces a
$BK_{r+1}$ we infer that $\wq(v, \ell)=2$ holds for some $\ell\in\{a, b, c\}$.
Together with the above inequality this shows that all vertices in $Q\cup\{a, b, c, x, y\}$
except for $\ell$ are red neighbours of $v$. Due to the symmetry between $a$ and $c$ 
we may suppose that $\ell\ne a$. Now $Q\cup\{a, v\}$ is the desired~$RK_{r-2}$ with
$xy$ in its neighbourhood.
\end{proof}
 
Finally, we deal with the alleged absence of wicked triangles.
 
\begin{lemma}\label{lem:wicked}
Let $G$ denote a $\{RK_r, BK_{r+1}\}$-free coloured graph of order $n$ such that
$\delta(G)>\tfrac{14r-24}{7r-5}n$.\begin{enumerate}[label=\rmlabel]
\item\label{it:wl1} If all blue edges of $G$ are secure, then every wicked triangle of $G$ 
possesses a green edge.
\item\label{it:wl2} If moreover the green edges of $G$ are secure as well, then $G$ contains no 
wicked triangles.
\end{enumerate}
\end{lemma}

\begin{proof} Let $V$ and $w$ be the vertex set and weight function of $G$. 
Arguing indirectly we let $(x, y, z)$ be a wicked triangle contradicting either 
of these two statements and such that subject to this $w(x, z)+w(y, z)$ is as large 
as possible. Set $\alpha=w(x, z)$ and $\beta=w(y, z)$. Notice that $\alpha, \beta\in \{0, 1\}$ 
and $xz$, $yz$ are secure. 
Consequently, there are two $(r-2)$-sets $A, B\subseteq V$ inducing red cliques such 
that~$x$,~$z$ belong to the common red neighbourhood of~$A$ while $y$, $z$ 
belong to the common red 
neighbourhood of $B$. Let us select these sets $A$ and $B$ in such a way that $k=|A\cap B|$
is maximal. Since $(A\cap B)\cup\{x, y\}$ is a red clique and $G$ is $RK_r$-free, we have 
\begin{align}\label{eq:k-small}
	k\le r-3\,.
\end{align}

\begin{figure}[ht]
    \centering      
\begin{tikzpicture}[scale = 1]
	\node[circle,fill=white,draw,minimum size=0.5cm, inner sep=0pt] (x1) at (-144:3) {$x$};
	\node[circle,fill=white,draw,minimum size=0.5cm, inner sep=0pt] (y1) at (-36:3) {$y$};
	\node[circle,fill=white,draw,minimum size=0.5cm, inner sep=0pt] (z) at (90:3) {$z$};
		
	\node[ellipse,fill=white,draw=red,thick,minimum width=1cm,minimum height=0.67cm, inner sep=0pt] (z1) at (0,0) { \textcolor{red}{{\large $RK_{k}$}}};
	\node[ellipse,fill=white,draw=red,thick,minimum width=1cm,minimum height=0.67cm, inner sep=0pt](x2) at (162:3) { \textcolor{red}{{\large $RK_{r-k-2}$}}};
	\node[ellipse,fill=white,draw=red,thick,minimum width=1cm,minimum height=0.67cm, inner sep=0pt] (y2) at (18:3) { \textcolor{red}{{\large $RK_{r-k-2}$}}};
	
	\node at (-5.3,0.9) {\textcolor{red}{\large $A\sm B$}};
	\node at (5.4,0.9) {\textcolor{red}{\large $B\sm A$}};
	\node at (-0.05,-0.7) {\textcolor{red}{\large $A\cap B$}};
	
	\path[black,thick]
		(x1) edge (z)
		(y1) edge (z)
		;
	\path[red,thick]
		(x1) edge (y1)
		(x1) edge (x2)
		(y1) edge (y2)
		(z) edge (x2)
		(z) edge (y2)
		(z1) edge (x1)
		(z1) edge (x2)
		(z1) edge (y1)
		(z1) edge (y2)
		(z1) edge (z);	
\end{tikzpicture}
        \caption{The sets $A$ and $B$. The black pairs are either blue or green.}
\end{figure}
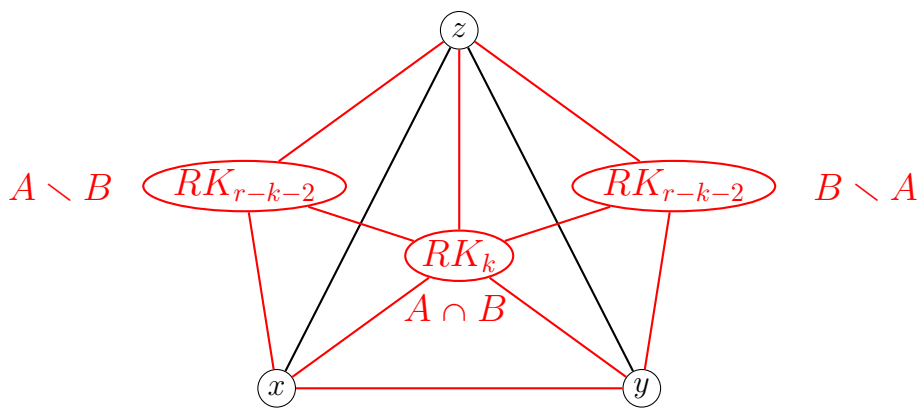

Notice that $x, y, z\not\in A\cup B$. Set $Q=(A\cup B)\cup\{x, y, z\}$,
and assign weights to the vertices in $Q$ by defining  
\[
	\gamma_q=
	\begin{cases}
		3+\alpha-\beta  & \text{ if } q\in A\setminus B \text{ or } q=x, \cr
		3+\beta-\alpha  & \text{ if } q\in B\setminus A \text{ or } q=y,  \cr
		7               & \text{ if } q\in A\cap B,                       \cr
		r-k+1           & \text{ if } q=z
	\end{cases}
\]
for $q\in Q$. So the total weight of the vertices in $Q$ is 
\[
	6(r-k-1)+7k+(r-k+1)=7r-5
\]
and by our standard argument there exists a vertex $v\in V$ with 
\begin{align}\label{eq:13-1}
	\sum_{q\in Q}\gamma_q\wq(v, q)\le 13\,.
\end{align}
It follows that
\begin{align}\label{eq:28}
	v\not\in (A\cap B) \text{ and there is no green edge from $v$ to } A\cap B\,.
\end{align}
Put 
\[
a=\sum_{q\in A\cup\{x\}} \wq(v, q)
\qand
b=\sum_{q\in B\cup\{y\}} \wq(v, q)
\]
and notice that~\eqref{eq:13-1} yields 
\begin{align}\label{eq:13-2}
	(3+\alpha-\beta)a+(3+\beta-\alpha)b+4\wq(v, z)\le 13\,,
\end{align}
because~\eqref{eq:k-small} implies $\gamma_z\ge 4$.

Since $A\cup\{x, v\}$ is not an $RK_r$, we have $a\ge 1$ and, similarly, $b\ge 1$. 
So~\eqref{eq:13-2} yields that
\begin{align}\label{eq:znotgreen}
	vz \text{ is either blue or red.}
\end{align}

If $\alpha=1$, i.e., if $xz$ is blue, then the fact that $A\cup\{v, x, z\}$ is not a $BK_{r+1}$
entails $a\ge 2$. Performing the same argument for $b$ we infer
\begin{align}\label{eq:a-big}
	a\ge 1+\alpha
	\qand
	b\ge 1+\beta\,.
\end{align}
Since $\alpha^2=\alpha$ and $\beta^2=\beta$, we have
\[
	13< 14+(1-\alpha)(1-\beta)+\alpha\beta
	= (3+\alpha-\beta)(4-\alpha-\beta+\alpha\beta)+(3+\beta-\alpha)(1+\beta)
\]
and by~\eqref{eq:13-2} and~\eqref{eq:a-big} this leads to 
\begin{equation}\label{eq:a-small}
	a\le 3-\alpha-\beta+\alpha\beta\,.
\end{equation}

Now assume there would exist two distinct vertices in $A\cup\{x\}$, say $s$ and $t$, that
fail to be red neighbours of $v$. Then $v\not\in A\cup\{x\}$ and, in particular, 
$v\not\in\{s, t\}$. So the maximality 
of $\alpha+\beta$ gives $w(v, s)+w(v, t)\le \alpha+\beta$, whence
\[
	a\ge \wq(v, s)+\wq(w, t)\ge 4-\alpha-\beta\,.
\]
In view of~\eqref{eq:a-small} this is only possible if $\alpha=\beta=1$ and 
the foregoing estimate on $a$ holds with equality. But then $A\cup\{v, x, z\}$
is a $BK_{r+1}$ in $G$, which is a contradiction. Therefore all but at most one vertex 
in $A\cup\{x\}$ are red neighbours of $v$. 

On the other hand, $A\cup\{v, x\}$ is not an $RK_r$ in $G$, so altogether we can conclude 
that there is a unique $a^*\in A\cup\{x\}$ such that $va^*$ is not red. Similarly, there 
is a unique $b^*\in B$ such that $vb^*$ is not red.

Next we suppose that $vz$ would be blue. Then, in particular, $v\not\in (A\cup B)$
and the combination of~\eqref{eq:13-2} and \eqref{eq:a-big} yields
\[
	13\ge 4+(1+\alpha)(3+\alpha-\beta)+(1+\beta)(3+\beta-\alpha)
	=10+3(\alpha+\beta)+(\alpha-\beta)^2\,,
\]
i.e., $\alpha=\beta=0$. So there is no wicked triangle with a blue edge and consequently 
there are only red edges from $v$ to $A\cup B$. Thus $a^*=x$ and $b^*=y$, for which reason 
$(x, y, v)$ is a wicked triangle. By the maximality of $\alpha+\beta$
it follows that $vx$ and $vy$ are green, i.e., that $a, b\ge 2$. But now we get a contradiction 
to~\eqref{eq:13-2}, which together with~\eqref{eq:znotgreen} proves that 
\begin{align}
	vz \text{ is red.}
\end{align}

Since $A\cup\{v, z\}$ cannot be an $RK_r$, it follows that $a^*\in A$ and, similarly, we have
$b^*\in B$. Owing to the uniqueness of $a^*$ and $b^*$ there are only two possibilities, namely
$a^*\in A\setminus B$ and $b^*\in B\setminus A$, or $a^*=b^*\in A\cap B$. If the former 
alternative would hold, then the sets $A\cup\{v\}\setminus\{a^*\}$ and 
$B\cup\{v\}\setminus\{b^*\}$ would contradict the maximality of $k$. So the only remaining 
case is that there is a member $u=a^*=b^*$ of $A\cap B$ such that $Q\setminus \{u\}$ is in the 
red neighbourhood of $v$. By~\eqref{eq:28} the edge $uv$ is blue. 
Since neither $A\cup \{v, x, z\}$
nor $B\cup\{v, y, z\}$ forms a $BK_{r+1}$, the edges $xz$ and $yz$ are green, i.e., $\alpha=\beta=0$. Let us recall that this means that there is no wicked triangle with a blue edge.

\begin{figure}[ht]
    \centering      
\begin{tikzpicture}[scale = 1]
	\node[circle,fill=white,draw,minimum size=0.5cm, inner sep=0pt] (x1) at (-144:3) {$x$};
	\node[circle,fill=white,draw,minimum size=0.5cm, inner sep=0pt] (y1) at (-36:3) {$y$};
	\node[circle,fill=white,draw,minimum size=0.5cm, inner sep=0pt] (z) at (90:3) {$z$};
		\node[circle,fill=white,draw,minimum size=0.5cm, inner sep=0pt] (u0) at (-1,-0.8) {$u$};
	\node[circle,fill=white,draw,minimum size=0.5cm, inner sep=0pt] (u1) at (1,-0.8) {$v$};

	\path[red,thick]
		(u1) edge (z)
		(u0) edge (z);
	\path[green,thick]
		(x1) edge (z)
		(y1) edge (z)
	;
	\node[ellipse,fill=white,draw=red,thick,minimum width=1cm,minimum height=0.67cm, inner sep=0pt] (z1) at (0,1) { \textcolor{red}{{\large $RK_{k-1}$}}};
	\node[ellipse,fill=white,draw=red,thick,minimum width=1cm,minimum height=0.67cm, inner sep=0pt](x2) at (162:3) { \textcolor{red}{{\large $RK_{r-k-2}$}}};
	\node[ellipse,fill=white,draw=red,thick,minimum width=1cm,minimum height=0.67cm, inner sep=0pt] (y2) at (18:3) { \textcolor{red}{{\large $RK_{r-k-2}$}}};
	\path[red,thick]
		(x1) edge (y1)
		(x1) edge (x2)
		(y1) edge (y2)
		(z) edge (x2)
		(z) edge (y2)
		(z1) edge (x1)
		(z1) edge (x2)
		(z1) edge (y1)
		(z1) edge (y2)
		(z1) edge (z)
		(u0) edge (x1)
		(u0) edge (x2)
		(u0) edge (z1)
		(u0) edge (y1)
		(u0) edge (y2)
		(u1) edge (x1)
		(u1) edge (x2)
		(u1) edge (z1)
		
		(u1) edge (y1)
		(u1) edge (y2)
	;
	\path[blue,thick]
		(u0) edge (u1)
	;
	
\end{tikzpicture}
        \caption{Current situation}
\end{figure} 

At this moment the weights $\gamma_q$ have done for us whatever they could do, and we proceed
by assigning new weights to the vertices in $Q$ and to $v$. To this end, we define
\[
	\eta_q=
	\begin{cases}
		1  & \text{ if } q\in (A \bigtriangleup B)\cup\{x, y, z\}, \cr
		2  & \text{ if } q\in\{u, v\}, \cr
		3  & \text{ if } q\in (A\cap B)\setminus\{u\}                       
	\end{cases}
\]
for $q\in Q\cup\{v\}$. By~\eqref{eq:k-small} the total weight $2r+k$ is at most $3(r-1)$
and thus there is a vertex $t$ with 
\begin{align}\label{eq:5-1}
	\sum_{q\in Q\cup\{v\}}\eta_q\wq(t, q)\le 5\,.
\end{align}

We will now analyse the set $T=\bigl\{q\in Q\cup\{v\}\colon \wq(q, t)=2\bigr\}$.
By~\eqref{eq:5-1} it needs to be disjoint to $A\cap B\sm\{u\}$. Suppose now that 
$v\in T$. Since the triangle $(u, t, v)$ cannot be wicked, it is not the case 
that $ut$ is a red edge, which in turn yields $\wq(u, t)+\wq(v, t)\ge 3$, 
contrary to~\eqref{eq:5-1}. This proves that $v\not\in T$ and by symmetry $u\not\in T$
holds as well. 

Now it follows from $G$ being $BK_{r+1}$-free that each of the four sets 
$(A\setminus B)\cup\{x\}$, $(A\setminus B)\cup\{z\}$, $(B\setminus A)\cup\{y\}$, and 
$(B\setminus A)\cup\{z\}$ contains a member of $T$. On the other hand~\eqref{eq:5-1} 
yields $|T|\le 2$. For these reasons, we have $T=\{a^*, b^*\}$
for two vertices~$a^*\in A\setminus B$ and $b^*\in B\setminus A$. 

Next we contend that $S=(Q\cup\{v\})\setminus T$ contains only red neighbours of $t$.
To see this, consider an arbitrary $s\in S$. In view of $s\not\in T$ the edge $st$ is either 
red or blue. Moreover, at least one of $a^*$ or $b^*$ is a red neighbour of $s$,
so suppose that $sa^*$ is red. Since $(s, a^*, t)$ cannot be a wicked triangle 
with a blue edge, it follows that $st$ is indeed red. 
  
Now the sets $A\cup\{t\}\setminus\{a^*\}$ and $B\cup\{t\}\setminus\{b^*\}$ contradict the maximality of $k$.
\end{proof}

We conclude this section by giving the proof of our second main result.

\begin{proof}[Proof of Theorem~\ref{thm:main-even}] Let $G'=(V, w')$ be an $\ccF_{2r}$-free
coloured graph with the property that ${w'(x, y)\ge w(x, y)}$ holds for all $x, y\in V$ 
and such that subject to this condition $e(G')$ is maximal. 
Then $G'$ is extremal and satisfies $\delta(G')\ge \delta(G)>\tfrac{14r-24}{7r-5} n$. 
Since every homomorphism from $G'$ to $RK_r^-$ is also a homomorphism from $G$ to $RK_r^-$, 
we may suppose for notational simplicity that $G'=G$, i.e., that $G$ itself is extremal.

Now by Lemma~\ref{lem:blue-secure} the blue edges of~$G$ are secure and 
Lemma~\ref{lem:wicked}\ref{it:wl1} informs us that~$G$ contains no blue wicked triangle.
This in turn implies in view of Lemma~\ref{lem:green-secure} that the green edges of $G$ 
are secure as well and, hence, Lemma~\ref{lem:wicked}\ref{it:wl2} is applicable, showing
that~$G$ contains no wicked triangles at all. This fact can be reformulated by saying that
the reflexive and symmetric relation ``$w(x, y)\in \{0, 1\}$'' is also 
transitive, i.e., an equivalence relation. Denote its (nonempty) equivalence classes 
by $A_1, \ldots, A_m$. We will suppose moreover that this indexing has been arranged in such 
a way that for some integer $s\in [0, m]$ each of the sets $A_1, \ldots, A_s$ spans at least 
one blue edge in $G$, whilst each of $A_{s+1}, \ldots, A_m$ forms a green clique. 

For every $i\in [m]$ we denote the minimum degree of the blue graph $G$ induces on $A_i$ 
by~$\alpha_i$. Notice that 
\[
	\tfrac{14r-24}{7r-5}n<\delta(G)\le 2(n-|A_i|)+\alpha_i
\]
holds for every $i\in [m]$, whence
\begin{equation}\label{eq:degAi}
	2|A_i|-\alpha_i< \tfrac{14}{7r-5}n\,.
\end{equation}
For $i\in [s+1, m]$ we have $\alpha_i=0$ and the previous inequality simplifies to 
$|A_i|<\tfrac{7}{7r-5}n$. If, however, $i\in [s]$, then the trivial bound
$\alpha_i<|A_i|$ leads to $|A_i|<\tfrac{14}{7r-5}n$. By adding these estimates
up we obtain
\[
	n=\sum_{i=1}^m |A_i|<\frac{14s+7(m-s)}{7r-5} n<\frac{m+s}{r-1} n\,,
\]
wherefore $m+s\ge r$. On the other hand, by taking arbitrary blue edges from each of 
$A_1, \ldots, A_s$ as well as arbitrary vertices from each of $A_{s+1}, \ldots, A_m$
we can construct a~$BK_{m+s}$ in $G$. So in view of Lemma~\ref{lem:bkr2} we must have 
$m+s=r$. Similar arguments shows that the blue graphs induced by $G$ 
on $A_1, \ldots, A_s$ are triangle-free. Moreover, one has $s\ge 1$, 
for otherwise $G$ would contain an $RK_r$.

Now for each $i\in [s]$ we find
\begin{align*}
	n&=|A_i|+\sum_{j\ne i}|A_j|<|A_i|+\frac{14(s-1)+7(m-s)}{7r-5} n \\
	 &=|A_i|+\frac{7(m+s-2)}{7r-5} n
	  =|A_i|+\frac{7r-14}{7r-5} n
\end{align*}
and, consequently, $|A_i|>\tfrac{9}{7r-5}n$. In combination with~\eqref{eq:degAi}
this leads to $2|A_i|-\alpha_i< \tfrac{14}{9}|A_i|$, i.e., $\alpha_i>\tfrac{4}{9}|A_i|$.
Since the blue graph $G$ induces on $A_i$ is triangle-free and $\tfrac{4}{9}>\tfrac{2}{5}$,
the case $r=2$ of Theorem~\ref{thm:AES} entails that 
this blue graph is bipartite. 

Thus for each $i\in [s]$ there is a partition $A_i=B_i\dcup C_i$ such that 
$B_i$ and $C_i$ are green cliques in~$G$. The structure we have thereby found in~$G$
may be regarded as a homomorphism from~$G$ to a coloured graph of order $m+s=r$
having a blue matching of size $s$ and otherwise red edges only. Due to $s\ge 1$
this proves Theorem~\ref{thm:main-even}. 		
\end{proof}


\begin{bibdiv}
\begin{biblist}

\bib{AES74}{article}{
   author={Andr{\'a}sfai, B.},
   author={Erd{\H{o}}s, P.},
   author={S{\'o}s, V. T.},
   title={On the connection between chromatic number, maximal clique and
   minimal degree of a graph},
   journal={Discrete Math.},
   volume={8},
   date={1974},
   pages={205--218},
   issn={0012-365X},
   review={\MR{0340075}},
}

\bib{BR}{article}{
	author={Bellmann, Louis},
	author={Reiher, Chr.}, 
	title={Tur\'an's Theorem for the Fano Plane}, 
	eprint={1804.07673},
	note={Combinatorica. To Appear},
}

\bib{BE76}{article}{
   author={Bollob{\'a}s, B{\'e}la},
   author={Erd{\H{o}}s, Paul},
   title={On a Ramsey-Tur\'an type problem},
   journal={J. Combinatorial Theory Ser. B},
   volume={21},
   date={1976},
   number={2},
   pages={166--168},
   review={\MR{0424613}},
}

\bib{Brandt}{article}{
   author={Brandt, Stephan},
   title={On the structure of graphs with bounded clique number},
   journal={Combinatorica},
   volume={23},
   date={2003},
   number={4},
   pages={693--696},
   issn={0209-9683},
   review={\MR{2047472}},
   doi={10.1007/s00493-003-0042-z},
}

\bib{DeFu00}{article}{
   author={De Caen, Dominique},
   author={F\"uredi, Zolt\'an},
   title={The maximum size of 3-uniform hypergraphs not containing a Fano
   plane},
   journal={J. Combin. Theory Ser. B},
   volume={78},
   date={2000},
   number={2},
   pages={274--276},
   issn={0095-8956},
   review={\MR{1750899}},
   doi={10.1006/jctb.1999.1938},
}

\bib{EHSS}{article}{
   author={Erd{\H{o}}s, P.},
   author={Hajnal, A.},
   author={S{\'o}s, Vera T.},
   author={Szemer{\'e}di, E.},
   title={More results on Ramsey-Tur\'an type problems},
   journal={Combinatorica},
   volume={3},
   date={1983},
   number={1},
   pages={69--81},
   issn={0209-9683},
   review={\MR{716422}},
   doi={10.1007/BF02579342},
}


\bib{ES69}{article}{
   author={Erd{\H{o}}s, P.},
   author={S{\'o}s, Vera T.},
   title={Some remarks on Ramsey's and Tur\'an's theorem},
   conference={
      title={Combinatorial theory and its applications, II},
      address={Proc. Colloq., Balatonf\"ured},
      date={1969},
   },
   book={
      publisher={North-Holland, Amsterdam},
   },
   date={1970},
   pages={395--404},
   review={\MR{0299512}},
}


\bib{FLZ15}{article}{
   author={Fox, Jacob},
   author={Loh, Po-Shen},
   author={Zhao, Yufei},
   title={The critical window for the classical Ramsey-Tur\'an problem},
   journal={Combinatorica},
   volume={35},
   date={2015},
   number={4},
   pages={435--476},
   issn={0209-9683},
   review={\MR{3386053}},
   doi={10.1007/s00493-014-3025-3},
}


\bib{FK02}{article}{
   author={F\"uredi, Zolt\'an},
   author={K\"undgen, Andr\'e},
   title={Tur\'an problems for integer-weighted graphs},
   journal={J. Graph Theory},
   volume={40},
   date={2002},
   number={4},
   pages={195--225},
   issn={0364-9024},
   review={\MR{1913847}},
   doi={10.1002/jgt.10012},
}

\bib{FuSi05}{article}{
   author={F\"uredi, Zolt\'an},
   author={Simonovits, Mikl\'os},
   title={Triple systems not containing a Fano configuration},
   journal={Combin. Probab. Comput.},
   volume={14},
   date={2005},
   number={4},
   pages={467--484},
   issn={0963-5483},
   review={\MR{2160414}},
   doi={10.1017/S0963548305006784},
}

\bib{KeMu12}{article}{
   author={Keevash, Peter},
   author={Mubayi, Dhruv},
   title={The Tur\'an number of $F_{3,3}$},
   journal={Combin. Probab. Comput.},
   volume={21},
   date={2012},
   number={3},
   pages={451--456},
   issn={0963-5483},
   review={\MR{2912791}},
   doi={10.1017/S0963548311000678},
}

\bib{KeSu05}{article}{
   author={Keevash, Peter},
   author={Sudakov, Benny},
   title={The Tur\'an number of the Fano plane},
   journal={Combinatorica},
   volume={25},
   date={2005},
   number={5},
   pages={561--574},
   issn={0209-9683},
   review={\MR{2176425}},
   doi={10.1007/s00493-005-0034-2},
}

\bib{LR-a}{article}{
	author={L\"uders, Clara Marie},
	author={Reiher, Chr.}, 
	title={The Ramsey-Tur\'an problem for cliques}, 
	eprint={1709.03352},
	note={Israel Journal of Mathematics. To Appear},
}

\bib{MuRo02}{article}{
   author={Mubayi, Dhruv},
   author={R\"odl, Vojt\^ech},
   title={On the Tur\'an number of triple systems},
   journal={J. Combin. Theory Ser. A},
   volume={100},
   date={2002},
   number={1},
   pages={136--152},
   issn={0097-3165},
   review={\MR{1932073}},
   doi={10.1006/jcta.2002.3284},
}

\bib{RoSi95}{article}{
   author={R\"odl, Vojt\v ech},
   author={Sidorenko, Alexander},
   title={On the jumping constant conjecture for multigraphs},
   journal={J. Combin. Theory Ser. A},
   volume={69},
   date={1995},
   number={2},
   pages={347--357},
   issn={0097-3165},
   review={\MR{1313901}},
   doi={10.1016/0097-3165(95)90057-8},
}

\bib{Si68}{article}{
   author={Simonovits, M.},
   title={A method for solving extremal problems in graph theory, stability
   problems},
   conference={
      title={Theory of Graphs},
      address={Proc. Colloq., Tihany},
      date={1966},
   },
   book={
      publisher={Academic Press, New York},
   },
   date={1968},
   pages={279--319},
   review={\MR{0233735}},
}

\bib{SS01}{article}{
   author={Simonovits, Mikl{\'o}s},
   author={S{\'o}s, Vera T.},
   title={Ramsey-Tur\'an theory},
   note={Combinatorics, graph theory, algorithms and applications},
   journal={Discrete Math.},
   volume={229},
   date={2001},
   number={1-3},
   pages={293--340},
   issn={0012-365X},
   review={\MR{1815611}},
   doi={10.1016/S0012-365X(00)00214-4},
}

\bib{Sz72}{article}{
   author={Szemer{\'e}di, Endre},
   title={On graphs containing no complete subgraph with $4$ vertices},
   language={Hungarian},
   journal={Mat. Lapok},
   volume={23},
   date={1972},
   pages={113--116 (1973)},
   issn={0025-519X},
   review={\MR{0351897}},
}
	
\bib{Sz78}{article}{
   author={Szemer{\'e}di, Endre},
   title={Regular partitions of graphs},
   language={English, with French summary},
   conference={
      title={Probl\`emes combinatoires et th\'eorie des graphes},
      address={Colloq. Internat. CNRS, Univ. Orsay, Orsay},
      date={1976},
   },
   book={
      series={Colloq. Internat. CNRS},
      volume={260},
      publisher={CNRS, Paris},
   },
   date={1978},
   pages={399--401},
   review={\MR{540024}},
}

	
\bib{Turan}{article}{
   author={Tur{\'a}n, Paul},
   title={Eine Extremalaufgabe aus der Graphentheorie},
   language={Hungarian, with German summary},
   journal={Mat. Fiz. Lapok},
   volume={48},
   date={1941},
   pages={436--452},
   review={\MR{0018405}},
}

\end{biblist}
\end{bibdiv}
\end{document}